\documentclass[a4paper,reqno]{amsart}
\usepackage{amssymb}
\usepackage[cp1251]{inputenc}
\usepackage[english,russian]{babel}

\PII{}
\makeatletter
\def\@setcopyright{\@empty}
\makeatother

\newcommand{\arr}[2]{{{#1}_1,\dots,{#1}_{#2}}}
\newcommand{\Si}[1]{(1-#1^2)}
\newcommand{\Co}[1]{\cos^4\frac{#1}2}
\newcommand{\Lmu}{L_{1,2}}
\newcommand{\allp}{1\le p\le\infty}
\newcommand{\Dx}{D_{x,2,2}}
\newcommand{\AD}{AD^r(p,\alpha)}
\newcommand{\K}{K_r(f,\delta)_{p,\alpha}}
\newcommand{\Ha}[2]{H^{#1}\left(#2\right)}
\newcommand{\Hd}[2]{H_{\delta}^{#1}\left(#2\right)}
\newcommand{\kap}{\varkappa(\delta)}
\newcommand{\sincosv}{%
  \left(\sin\frac v2\right)^{-1}
  \left(\cos\frac v2\right)^{-9}}
\newcommand{\sincosu}{%
  \left(\sin\frac u2\right)\left(\cos\frac u2\right)^9}
\newcommand{\Px}[1]{P_{#1}^{(2,2)}}
\newcommand{\E}{E_n(f)_{p,\alpha}}
\newcommand{\Epar}[2]{E_{#1}\left(#2\right)_{p,\alpha}}
\newcommand{\T}[3]{\tau_{#1}^{#2}\left(#3\right)}
\newcommand{\hatT}[3]{\hat\tau_{#1}^{#2}\left(#3\right)}
\newcommand{\Dl}[3]{\Delta_{#1}^{#2}\left(#3\right)}
\newcommand{\w}{\hat\omega_r(f,\delta)_{p,\alpha}}
\newcommand{\wpar}[1]{\hat\omega_r\left(#1\right)_{p,\alpha}}
\newcommand{\Kpar}[1]{K_r\left(#1\right)_{p,\alpha}}
\newcommand{\norm}[1]{\left\|#1\right\|_{p,\alpha}}
\newcommand{\Dy}{D_{y,0,4}}
\newcommand{\Py}[1]{P_{#1}^{(0,4)}}
\newcommand{\Lp}{L_{p,\alpha}}
\newcommand{\gd}{g_{\delta,r}(x)}

\newcommand{\numericset}[1]{\mathbb #1}
\newcommand{\numN}{\numericset N}

\newtheorem{thm}{Теорема}[section]
\newtheorem{lmm}{Лемма}[section]
\newtheorem*{cor}{Следствие}

\newcounter{const}
\numberwithin{const}{thm}
\numberwithin{const}{lmm}

\makeatletter
\newcommand{\Cn}[1][]{%
  \stepcounter{const}C_{\theconst}%
  \@ifnotempty{#1}{%
    \newcounter{#1}\setcounter{#1}{\arabic{const}}}}
\makeatother

\newcommand{\lastC}{C_{\theconst}}
\newcommand{\prevC}[1][1]{%
	{\countdef\n=255
	 \n=\theconst
	 \advance\n by-#1
	 C_{\number\n}}}

\numberwithin{equation}{section}

\renewcommand{\theconst}{\arabic{const}}

\begin{document}

\title[О связи между $r$-ым обобщенным модулем гладкости\dots]
	{О связи между $r$-ым обобщенным модулем гладкости
		и наилучшими приближениями алгебраическими многочленами}
\author{M.~K.\ Потапов}
\address{M.~K.\ Потапов\\
	Механико-математический факультет\\
	Московский Государственный Университет им.\ Ломоносова\\
	Москва 117234\\
	Россия}
\author{Ф.~M.\ Бериша}
\address{Ф.~M.\ Бериша\\
	Механико-математический факультет\\
	Московский Государственный Университет им.\ Ломоносова\\
	Москва 117234\\
	Россия}
\curraddr{F.~M.\ Berisha\\
	Faculty of Mathematics and Sciences\\
	University of Prishtina\\
	N\"ena Terez\"e~5\\
	10000 Prishtina\\
	Kosovo}
	\email{faton.berisha@uni-pr.edu}

\thanks{Работа выполнена
	при поддержке Российского Фонда Фундаментальных Исследования
	(грант No.\ 97--01--00010)
	и программы поддержки ведущих научных школ
	(грант No.\ 96/97--15--96073).}

\keywords{\foreignlanguage{english}{%
  Generalised modulus of smoothness,
	asymmetric operator of generalised translation,
	Jackson theorem, converse theorem,
	best approximations by algebraic polynomials}}
\subjclass{Primary 41A35, Secondary 41A50, 42A16. (UDK 517.5.)}
\date{}

\begin{abstract}
	В данной работе вводится несимметричный оператор обобщенного сдвига,
	с его помощью определяется обобщенный модуль гладкости
	и для него доказывается прямая и обратная теоремы
	теории приближений.
\end{abstract}

\maketitle

\begin{otherlanguage}{english}
\begin{abstract}
	In this paper an asymmetrical operator of generalised translation
	is introduced,
	the generalised modulus of smoothness is defined by its means
	and the direct and inverse theorems in approximation theory
	are proved for that modulus.
\end{abstract}
\end{otherlanguage}

\section{Введение}

Для $2\pi$--периодических функций хорошо известны связи
между $r$-ым обычным модулем гладкости
$\omega_r(f,\delta)_{p*}$
функции $f\in L_{p*}$
с ее наилучшими приближениями $E_n(f)_{p*}$
тригонометрическими полиномами
порядка не выше чем, $n-1$:
\begin{equation}\label{eq:jackson-periodic}
	\Cn E_n(f)_{p*}\le \omega_r\left(f,\frac1n\right)_{p*}
	\le\Cn\frac1{n^r}\sum_{\nu=1}^n \nu^{r-1}E_\nu(f)_{p*},
\end{equation}
где~$\prevC$ и~$\lastC$ --- положительные постоянные,
не зависящие от~$f$ и~$n$ $(n\in\numN)$.

При рассмотрении непериодических функций,
заданных на конечном отрезке вещественной оси,
уже не удается получить такие же связи
между обычными модулями гладкости этих функций
и их наилучшими приближениями алгебраическими многочленами.

Однако полная аналогия с $2\pi$ периодическим случаем
имеет место тогда,
когда обычный модуль гладкости
заменен обобщенным модулем гладкости
(см.\ например~\cite{pawelke:acta-72,butzer-s-w:c-80,
potapov:vestnik-83,potapov-f:trudy-85}).

В этой работе
доказывается аналог неравенства~\eqref{eq:jackson-periodic}
для $r$-го обобщенного модуля гладкости,
определяемого при помощи одного несимметричного оператора обобщенного сдвига.

\section{Определение обобщенного модуля гладкости}

Обозначим через $L_p$, $1\le p<\infty$, множество функций~$f$,
измеримых по Лебегу и суммируемых в~$p$-й степень
на отрезке $[-1,1]$,
а через~$L_\infty$ обозначим множество функций,
непрерывных на отрезке $[-1,1]$,
причем
\begin{displaymath}
	\|f\|_p=
	\begin{cases}
	\left(\int_{-1}^1|f(x)|^p\,dx\right)^{1/p},
	  & \text{если $1\le p<\infty$},\\
	\max_{-1\le x\le1}|f(x)|,
	  & \text{если $p=\infty$}.
	\end{cases}
\end{displaymath}

Через~$\Lp$ обозначим множество функций~$f$,
таких, что $f(x)\Si{x}^\alpha\in L_p$,
причем
\begin{displaymath}
	\norm f=\left\|f(x)\Si{x}^\alpha\right\|_p.
\end{displaymath}

Через $\E$ обозначим наилучшее приближение функций~$f$
при помощи алгебраических многочленов степени не выше, чем $n-1$,
в метрике~$\Lp$,
т.е.
\begin{displaymath}
	\E=\inf_{P_n}\norm{f-P_n},
\end{displaymath}
где~$P_n$ --- алгебраический многочлен
степени не выше,
чем $n-1$.

Для суммируемой функции~$f$ введем оператор обобщенного сдвига
по правилу
\begin{displaymath}
	\hatT t{}{f,x}=\frac1 {\pi\Si x\Co t}
	\int_0^\pi B_{\cos t}(x,\cos\varphi,R)
	f(R)\,d\varphi,
\end{displaymath}
где
\begin{align}\label{eq:R-B}
	R 			&=x\cos t-\sqrt{1-x^2}\sin t\cos\varphi, \notag\\
	B_y(x,z,R) 	&=2\Bigl(\sqrt{1-x^2}y+xz\sqrt{1-y^2}\\
				& \quad
		+\sqrt{1-x^2}(1-y)\Si z\Bigr)^2-\Si R. \notag
\end{align}

При помощи этого оператора обобщенного сдвига
определим~$r$-ю обобщенную разность по правилу
\begin{align*}     
	\Dl t1{f,x} &=\Dl t{}{f,x}=\hatT t{}{f,x}-f(x),\\
	\Dl{\arr t r}r{f,x} &=\Dl{t_r}{}{\Dl{\arr t{r-1}}{r-1}{f,x},x}
		\quad (r=2,3,\dots).
\end{align*}
и, для функции $f\in\Lp$,
$r$-й обобщенный модуль гладкости
по правилу
\begin{displaymath}
	\w=\sup_{\substack{|t_j|\le\delta\\j=1,2,\dots,r}}
	  \norm{\Dl{\arr t r}r{f,x}}
	\quad(r=1,2,\dotsc).
\end{displaymath}

Полагая $y=\cos t$, $z=\cos\varphi$ в операторе
$\hatT t{}{f,x}$,
обозначим его через $\T y{}{f,x}$ и запишем в виде
\begin{displaymath}
	\T y{}{f,x}=\frac4{\pi\Si x(1+y)^2}
	\int_{-1}^1 B_y(x,z,R)f(R)\frac{dz}{\sqrt{1-z^2}},
\end{displaymath}
где~$R$ и~$B_y(x,z,R)$ определены формулами~\eqref{eq:R-B}.

Определим $r$-й оператор обобщенного сдвига по правилу
\begin{align*}
	\T{y}1{f,x} &=\T y{}{f,x},\\
	\T{\arr y r}r{f,x} &=\T{y_r}{}{\T{\arr t{r-1}}{r-1}{f,x},x}
		\quad (r=2,3,\dots).
\end{align*}

Обозначим через~$D_{x,\nu,\mu}$ оператор дифференцирования,
определяемый по правилу
\begin{displaymath}
	D_{x,\nu,\mu}=\Si{x}\frac{d^2}{dx^2}
	+(\mu-\nu-(\nu+\mu+2)x)\frac d{dx}.
\end{displaymath}
Ясно, что
\begin{displaymath}
	D_{x,\nu,\mu}=(1-x)^{-\nu}(1+x)^{-\mu}\frac d{dx}
	(1-x)^{\nu+1}(1+x)^{\mu+1}\frac d{dx}.
\end{displaymath}

Будем обозначать
\begin{align*}
	D_{x,\nu,\mu}^1f(x) &=D_{x,\nu,\mu}f(x),\\
	D_{x,\nu,\mu}^r f(x)
		&=D_{x,\nu,\mu}(D_{x,\nu,\mu}^{r-1} f(x))
		\quad (r=2,3,\dotsc).
\end{align*}

Будем писать, что $f(x)\in\AD$,
если $f\in\Lp$, $f(x)$ имеет на каждом отрезке $[a,b]\subset(-1,1)$
абсолютно непрерывную $2r-1$ производную
$\frac{d^{2r-1}}{dx^{2r-1}}f(x)$
и $\Dx^l f(x)\in\Lp$ $(l=0,1,\dots,r)$.

Обозначим через
\begin{displaymath}
	\K=\inf_{g\in\AD}\left(\norm{f-g}
	+\delta^{2r}\norm{\Dx^r g(x)}\right)
\end{displaymath}
$K$--функционал типа Петре,
интерполирующий между пространствами $\Lp$ и $\AD$.

Для $f\in\Lmu$ обозначим через $\Ha{}{f,x}$ и $\Hd{}{f,x}$
следующие операторы
\begin{displaymath}
	\Ha{}{f,x}=-\int_0^x\Si{y}^{-3}
	\int_y^1\left(f(z)-\frac{c_1}{c_0}\right)\Si{z}^2\,dz\,dy,
\end{displaymath}
где $c_1=\int_{-1}^1f(z)\Si{z}^2\,dz$,
$c_0=\int_{-1}^1\Si{z}^2\,dz$;
и
\begin{multline*}
	\Hd{}{f,x}=\frac1{\kap}\int_0^\delta\sincosv\\
	\times\int_0^v\hatT u{}{f,x}\sincosu\,du\,dv,
\end{multline*}
где
\begin{displaymath}
	\kap=\int_0^\delta\sincosv\int_0^v\sincosu\,du\,dv.
\end{displaymath}

Определим $r$-ую степень оператора~$H$ по правилу
\begin{multline*}
	\begin{aligned}
		\Ha1{f,x}  &=\Ha{}{f,x},\\
		\Ha r{f,x} &=\Ha{}{\Ha{r-1}{f,x},x}=-\int_0^x\Si{y}^{-3}
	\end{aligned}
	\\
		\times\int_y^1
			\left(
				\Ha{r-1}{f,z}-\frac{c_r}{c_0}
			\right)
			\Si{z}^2\,dz\,dy \quad (r=2,3,\dotsc),
\end{multline*}
где $c_r=\int_{-1}^1\Ha{r-1}{f,z}\Si{z}^2\,dz$;
и $r$-ую степень оператора~$H_\delta$ по правилу
\begin{align*}
	\Hd1{f,x}  &=\Hd{}{f,x},\\
	\Hd r{f,x} &=\Hd{}{\Hd{r-1}{f,x},x} \quad(r=2,3,\dotsc).
\end{align*}

Будем обозначать через $P_n^{(\nu,\mu)}(x)$ $(n=0,1,\dotsc)$
многочлены Якоби,
т.е. многочлены степени~$n$
ортогональные друг другу с весом $(1-x)^\nu(1+x)^\mu$
на отрезке $[-1,1]$
и нормированные условием
$P_n^{(\nu,\mu)}(1)=1$ $(n=0,1,\dotsc)$.

Через~$a_n(f)$ обозначим коэффициенты Фурье--Якоби
функции $f\in\Lmu$
по системе многочленов Якоби
$\left\{\Px n(x)\right\}_{n=0}^\infty$,
т.е.
\begin{displaymath}
	a_n(f)=\int_{-1}^1f(x)\Px n(x)\Si{x}^2\,dx
	\quad
	(n=0,1,\dotsc).
\end{displaymath}

\section{Вспомогательные утверждения}

\begin{lmm}[\cite{potapov:vestnik-60a}]\label{lm:bernshtein-markov}
	Пусть~$P_n(x)$--алгебраический многочлен
	степени не выше, чем $n-1$,
	$\allp$, $\rho\ge0$;
	\begin{alignat*}2
		\alpha &>-\frac1p	&\quad &\text{при $1\le p<\infty$},\\
		\alpha &\ge0 		&\quad &\text{при $p=\infty$}.
	\end{alignat*}
	Тогда справедливы неравенства
	\begin{gather*}
		\left\|P'_n(x)\right\|_{p,\alpha+\frac12}
			\le\Cn n\norm{P_n},\\
		\norm{P_n}
			\le\Cn n^{2\rho}\left\|P_n\right\|_{p,\alpha+\rho},
	\end{gather*}
	где постоянные~$\prevC$ и~$\lastC$
	не зависят от~$n$ $(n\in\numN)$.
\end{lmm}

\begin{cor}
	Пусть~$P_n(x)$ --- алгебраический многочлен
	степени не выше, чем $n-1$,
	$\allp$;
	\begin{alignat*}2
		\alpha &>-\frac1p	&\quad &\text{при $1\le p<\infty$},\\
		\alpha &\ge0 		&\quad &\text{при $p=\infty$}.
	\end{alignat*}
	Тогда
	\begin{displaymath}
		\norm{\Dx{P_n(x)}}\le Cn^2\norm{P_n(x)},
	\end{displaymath}
	где постоянная~$C$ не зависит от~$n$ $(n\in\numN)$.
\end{cor}

\begin{proof}
	Так как,
	\begin{displaymath}
		\norm{\Dx{P_n(x)}}\le\left\|P''_n(x)\right\|_{p,\alpha+1}
		+6\norm{P'_n(x)},
	\end{displaymath}
	то, применяя дважды лемму~\ref{lm:bernshtein-markov},
	получаем утверждение следствия.
\end{proof}

\begin{lmm}[\cite{p-berisha:anal-99}]\label{lm:properties-tau}
	Оператор~$\tau_y$ обладает следующими свойствами
	\begin{enumerate}
	\item\label{it:properties-tau-1}
		Оператор $\T y{}{f,x}$
		линеен по~$f$;
	\item\label{it:properties-tau-2}
		$\T1{}{f,x}=f(x)$;
	\item\label{it:properties-tau-3}
		$\T y{}{\Px n,x}=\Px n(x)\Py n(y)$ $(n=0,1,\dotsc)$;
	\item\label{it:properties-tau-4}
		$\T y{}{1,x}=1$;
	\item\label{it:properties-tau-6}
		$a_n\left(\T y{}{f,x}\right)=a_n(f)\Py n(y)$
		$(n=0,1,\dotsc)$.
	\end{enumerate}
\end{lmm}

\begin{lmm}[\cite{p-berisha:anal-99}]\label{lm:properties-tau-5}
	Пусть $g(x)\T y{}{f,x}\in\Lmu$ для любого~$y$.
	Тогда справедливо равенство
	\begin{displaymath}
		\int_{-1}^1f(x)\T y{}{g,x}\Si{x}^2\,dx
		=\int_{-1}^1g(x)\T y{}{f,x}\Si{x}^2\,dx.
	\end{displaymath}
\end{lmm}

\begin{lmm}[\cite{p-berisha:anal-99}]\label{lm:bound-tau}
	Пусть даны числа~$p$ и~$\alpha$ такие,
	что $\allp$;
	\begin{alignat*}2
		\frac12      &<\alpha\le1
		  &\quad &\text{при $p=1$},\\
		1-\frac1{2p} &<\alpha<\frac32-\frac1{2p}
		  &\quad &\text{при $1<p<\infty$},\\
		1            &\le\alpha<\frac32
		  &\quad &\text{при $p=\infty$}.
	\end{alignat*}
	Пусть $f\in\Lp$.
	Тогда 
	справедливо неравенство
	\begin{displaymath}
		\norm{\hatT t{}{f,x}}\le C\frac1{\Co t}\norm f,
	\end{displaymath}
	где постоянная~$C$ не зависит от~$f$ и~$t$.
\end{lmm}

\begin{lmm}[\cite{p-berisha:anal-99}]\label{lm:tauuDx}
	Пусть функция~$f(x)$ имеет абсолютно непрерывную
	на каждом отрезке $[a,b]\subset(-1,1)$
	производную~$f'(x)$.
	Тогда для почти всех $x\in[-1,1]$
	выполнены следующие равенства
	\begin{multline*}
		\hatT t{}{f,x}-f(x)\\
		=\int_0^t\sincosv\int_0^v\hatT u{}{\Dx f,x}\sincosu\,du\,dv
	\end{multline*}
	и
	\begin{multline*}
		\hatT t{}{f,x}-\hatT {\pi/2}{}{f,x}\\
		=-\int_{\pi/2}^t\sincosv
		  \int_v^\pi\hatT u{}{\Dx f,x}\sincosu\,du\,dv.
	\end{multline*}
\end{lmm}

\begin{lmm}\label{lm:Dtau}
	Пусть функция~$f(x)$ имеет на каждом отрезке
	$[a,b]\subset(-1,1)$
	абсолютно непрерывную $2r-1$ производную
	$\frac{d^{2r-1}}{dx^{2r-1}}f(x)$.
	Тогда
	\begin{enumerate}
	\item\label{it:Dtau1}
		При фиксированном~$y$
		функция $\T y{}{f,x}$ имеет на каждом отрезке
		$[c,d]\subset(-1,1)$
		абсолютно непрерывную $2r-1$ производную по~$x$
		$\frac{d^{2r-1}}{dx^{2r-1}}\T y{}{f,x}$.
	\item\label{it:Dtau2}
		При фиксированном~$x$
		функция $\T y{}{f,x}$ имеет на каждом отрезке
		$[c,d]\subset(-1,1)$
		абсолютно непрерывную $2r-1$ производную по~$y$
		$\frac{d^{2r-1}}{dy^{2r-1}}\T y{}{f,x}$.
	\item\label{it:Dtau3}
		Для почти всех~$x$ и~$y$
		справедливы равенства
		\begin{displaymath}
			\T y{}{\Dx f,x}=\Dx\T y{}{f,x}=\Dy\T y{}{f,x}.
		\end{displaymath}
	\end{enumerate}
\end{lmm}

\begin{proof}
	Докажем утверждение~\ref{it:Dtau1}.
	Для $r=1$ оно доказано в работе~\cite{p-berisha:anal-99}.
	Обозначим
	\begin{displaymath}
		\varphi(x)=\frac{B_y(x,z,R)}{\Si{x}(1+y)^2\sqrt{1-z^2}}f(R),
	\end{displaymath}
	где~$R$ и~$B_y(x,z,R)$ даны формулами~\eqref{eq:R-B}.
	Применяя индукцию, можно доказать,
	что для $l=1,\dots,2r-1$ имеем
	\begin{multline*}
		\frac{d^l}{dx^l}\varphi(x)=\varphi^{(l)}(x)\\
		=\frac1{(1+y)^2\sqrt{1-z^2}}\sum_{k=0}^l\binom rk
			\left(\frac{d^{l-k}}{dx^{l-k}}\frac{B_y(x,z,R)}{1-x^2}\right)
			\frac{d^k}{dx^k}f(R),
	\end{multline*}
	где
	\begin{displaymath}
		\frac{d^k}{dx^k}f(R)=\sum_{\nu=1}^k\frac{d^\nu f(R)}{dR^\nu}
			\sum_{
			   \substack{\mu_1\ge\dots\ge\mu_\nu\ge0\\
				  \mu_1+\dots+\mu_\nu=k}}
			   \alpha_k\prod_{j=1}^k\frac{d^{\mu_j}R}{dx^{\mu_j}}.
	\end{displaymath}
	
	Аналогичным рассуждением как в случае $r=1$
	(см.~\cite{p-berisha:anal-99})
	доказывается,
	что функция~$\varphi^{(l)}(x)$ абсолютно непрерывна
	на каждом отрезке
	$[c,d]\subset(-1,1)$ $(l=1,\dots,2r-1)$.
	Воспользовавшись теоремой Лебега,
	при фиксированных~$y$ и~$z$,
	получаем, что существует конечная производная
	$\frac{d^{2r-1}}{dx^{2r-1}}\T y{}{f,x}$
	--- абсолютно непрерывная на каждом отрезке $[c,d]\subset(-1,1)$.
	
	Используя симметричность~$R$ по~$x$ и~$y$,
	аналогичным рассуждением
	доказывается абсолютная непрерывность функции
	$\frac{d^{2r-1}}{dy^{2r-1}}\T y{}{f,x}$
	при фиксированном~$x$.
	
	Утверждение~\ref{it:Dtau3}
	доказано в работе~\cite{p-berisha:anal-99}.
	
	Лемма~\ref{lm:Dtau} доказана.
\end{proof}

\begin{lmm}\label{lm:cr:Dtau}
	Пусть функция~$f(x)$ имеет
	на каждом отрезке $[a,b]\subset(-1,1)$
	абсолютно непрерывную $2l-1$ производную
	$\frac{d^{2l-1}}{dx^{2l-1}}f(x)$.
	Тогда для почти всех~$x$ и~$y$
	справедливы равенства
	\begin{displaymath}
		\T{\arr y r}r{\Dx^l f,x}=\Dx^l\T{\arr y r}r{f,x}
		\quad (r=1,2,\dotsc).
	\end{displaymath}
\end{lmm}

\begin{proof}
	При $r=l=1$ равенство леммы
	следует из леммы~\ref{lm:Dtau}.
	
	Пусть $l\ge2$, $r=1$.
	Ясно, что $\Dx^{l-1}f(x)$ имеет абсолютно непрерывную
	на каждом отрезке $[a,b]\subset(-1,1)$
	производную $\frac d{dx}\Dx^{l-1}f(x)$.
	Поэтому, из леммы~\ref{lm:Dtau} следует,
	что
	\begin{displaymath}
		\T{y_1}{}{\Dx^l f,x}=\Dx\T{y_1}{}{\Dx^{l-1}f,x}.
	\end{displaymath}
	Применяя это равенство $l$~раз получим, что
	\begin{displaymath}
		\T{y_1}{}{\Dx^l f,x}=\Dx^l\T{y_1}{}{f,x}.
	\end{displaymath}
	Значит, равенство леммы справедливо
	при любых $l\in\numN$ и $r=1$.
	
	Теперь, применяя индукцию,
	нетрудно доказать утверждение леммы
	при любых натуральных~$r$ и~$l$.
	
	Лемма~\ref{lm:cr:Dtau} доказана.
\end{proof}

\begin{lmm}\label{lm:w-Dx}
	Пусть даны числа~$p$, $\alpha$ и~$r$ такие,
	что $\allp$,
	$r\in\numN$;
	\begin{alignat*}2
		\frac12      &<\alpha\le1
		  &\quad &\text{при $p=1$},\\
		1-\frac1{2p} &<\alpha<\frac32-\frac1{2p}
		  &\quad &\text{при $1<p<\infty$},\\
		1            &\le\alpha<\frac32
		  &\quad &\text{при $p=\infty$}.
	\end{alignat*}
	Пусть $g(x)\in\AD$.
	Для $0\le\delta<\pi$ справедливо неравенство
	\begin{displaymath}
		\wpar{g,\delta}\le C\frac1{(\cos\delta/2)^{4r}}\delta^{2r}
		\norm{\Dx^r g(x)},
	\end{displaymath}
	где постоянная~$C$ не зависит от~$g$ и~$\delta$.
\end{lmm}

\begin{proof}
	Докажем сначала,
	что при $|t_i|<\pi$ $(i=1,\dots,r)$
	справедливо неравенство
	\begin{equation}\label{eq:Dl-Dx}
		\norm{\Dl{\arr t r}r{g,x}}
		\le\frac\Cn{\prod_{i=1}^r \Co{t_i}}
			t_1^2\dots t_r^2\norm{\Dx^r g(x)},
	\end{equation}
	где постоянная~$\lastC$
	не зависит от~$g$ и~$t_i$ $(i=1,\dots,r)$.
	
	Для $r=1$ неравенство~\eqref{eq:Dl-Dx}
	доказано в работе~\cite{p-berisha:anal-99}.
	
	Предположим, что справедливо неравенство~\eqref{eq:Dl-Dx}.
	Так как и при доказательстве для $r=1$,
	только взяв~$\Dl{\arr t r}r{g,x}$ вместо~$g(x)$,
	учитывая, что из лемм~\ref{lm:Dtau} и~\ref{lm:bound-tau}
	вытекает $\Dl{\arr t r}r{g,x}\in AD^{r+1}(p,\alpha)$,
	получим
	\begin{displaymath}
		\norm{\Dl{\arr t{r+1}}{r+1}{g,x}}
		\le\frac\Cn{\Co{t_{r+1}}}t_{r+1}^2
		  \norm{\Dx\Dl{\arr t r}r{g,x}}.
	\end{displaymath}
	Применяя лемму~\ref{lm:cr:Dtau} и предположение леммы,
	получаем, что
	\begin{displaymath}
		\norm{\Dl{\arr t{r+1}}{r+1}{g,x}}
		\le\frac\Cn{\Co{t_{r+1}}}t_{r+1}^2\norm{\Dl{\arr t r}r{\Dx g,x}}.
	\end{displaymath}
	
	На основании индукции, учитывая, что $\Dx g(x)\in\AD$,
	получаем, что неравенство~\eqref{eq:Dl-Dx} справедливо.
	
	Переходя в~\eqref{eq:Dl-Dx} к точной верхней грани
	по всем~$t_i$, $|t_i|<\delta$ $(i=1,\dots,r)$,
	получим неравенство леммы.
	
	Лемма~\ref{lm:w-Dx} доказана.
\end{proof}

\begin{lmm}\label{lm:H-Lp}
	Пусть даны числа~$p$ и~$\alpha$ такие,
	что $\allp$;
	\begin{alignat*}2
		-1 		&<\alpha\le2 		 &\quad &\text{при $p=1$},\\
		-\frac1p &<\alpha<3-\frac1p &\quad &\text{при $1<p<\infty$},\\
		0 			&\le\alpha<3 		 &\quad &\text{при $p=\infty$}.
	\end{alignat*}
	Тогда если $f\in\Lp$,
	то $\Ha{}{f,x}\in\Lp$.
\end{lmm}

\begin{proof}
	Нетрудно заметить, что при условиях леммы $f\in\Lmu$.
	Значит существует $\Ha{}{f,x}$.
	
	Для $1\le p<\infty$ обозначим
	\begin{displaymath}
		I=\norm{\Ha{}{f,x}}^p
		=\int_{-1}^1\Si{x}^{p\alpha}|\Ha{}{f,x}|^p\,dx.
	\end{displaymath}
	
	Пусть $p=1$.
	Рассмотрим
	\begin{multline*}
		I_1=\int_0^1\Si{x}^\alpha|\Ha{}{f,x}|\,dx\\
		\le\int_0^1\Si{x}^\alpha\int_0^x\Si{y}^{-3}
			\int_y^1\Si{z}^2
			  \left|f(z)-\frac{c_1}{c_0}\right|\,dz\,dy\,dx.
	\end{multline*}
	Ясно, что
	\begin{displaymath}
		I_1\le\Cn\int_0^1(1-x)^\alpha\int_0^x(1-y)^{-3}
		\int_y^1(1-z)^2\left|f(z)-\frac{c_1}{c_0}\right|\,dz\,dy\,dx.
	\end{displaymath}
	Поменяв пределы интегрирования, учитывая,
	что $-1<\alpha\le2$,
	имеем
	\begin{multline*}
		I_1\le\lastC\int_0^1(1-z)^2\left|f(z)-\frac{c_1}{c_0}\right|
			  \int_0^z(1-y)^{-3}\int_y^1(1-x)^\alpha\,dx\,dy\,dz\\
		=\frac{\lastC}{\alpha+1}\int_0^1(1-z)^2
			  \left|f(z)-\frac{c_1}{c_0}\right|
			  \int_0^z(1-y)^{\alpha-2}\,dy\,dz\\
		\le\Cn\int_0^1(1-z)^\alpha z
			\left|f(z)-\frac{c_1}{c_0}\right|\,dz
		\le\lastC\left\|f-\frac{c_1}{c_0}\right\|_{1,\alpha}.
	\end{multline*}
	Поскольку $f\in L_{1,\alpha}$ и $\alpha>-1$,
	то
	\begin{displaymath}
		I_1<\infty.
	\end{displaymath}
	
	Рассмотрим
	\begin{multline*}
		I_2=\int_{-1}^0\Si{x}^\alpha|\Ha{}{f,x}|\,dx\\
		=\int_{-1}^0\Si{x}^\alpha
		  \left|\int_0^x\Si{y}^{-3}
			  \int_y^1\Si{z}^2
				\left(f(z)-\frac{c_1}{c_0}\right)\,dz\,dy
			\right|\,dx.
	\end{multline*}
	Из определения~$c_1$ и~$c_0$ следует, что
	\begin{displaymath}
		\int_{-1}^1\Si{z}^2\left(f(z)-\frac{c_1}{c_0}\right)\,dz=0.
	\end{displaymath}
	Поэтому
	\begin{equation}\label{eq:y-1-fz}
		\int_y^1\Si{z}^2\left(f(z)-\frac{c_1}{c_0}\right)\,dz
		=-\int_{-1}^y\Si{z}^2
		  \left(f(z)-\frac{c_1}{c_0}\right)\,dz.
	\end{equation}
	Отсюда вытекает, что
	\begin{displaymath}
		I_2\le\Cn\int_{-1}^0(1+x)^\alpha\int_x^0(1+y)^{-3}
		  \int_{-1}^y(1+z)^2
			\left|f(z)-\frac{c_1}{c_0}\right|\,dz\,dy\,dx.
	\end{displaymath}
	Меняя пределы интегрирования получаем
	\begin{displaymath}
		I_2\le\lastC\int_{-1}^0(1+z)^2
		  \left|f(z)-\frac{c_1}{c_0}\right|
		  \int_z^0(1+y)^{-3}\int_{-1}^y (1+x)^\alpha\,dx\,dy\,dz.
	\end{displaymath}
	Отсюда,
	аналогичным рассуждением как в предыдущем случае
	получаем,
	что
	\begin{displaymath}
		I_2<\infty.
	\end{displaymath}
	
	Таким образом, при $p=1$ доказано, что
	\begin{displaymath}
		I=I_1+I_2<\infty.
	\end{displaymath}
	Значит,
	$\Ha{}{f,x}\in L_{1,\alpha}$.
	
	Пусть $1<p<\infty$.
	Имеем
	\begin{displaymath}
		|\Ha{}{f,x}|
		\le\int_0^x\Si{y}^{-3}
			\int_y^1\Si{z}^2\left|f(z)-\frac{c_1}{c_0}\right|\,dz\,dy.
	\end{displaymath}
	
	Рассмотрим
	\begin{displaymath}
		I_3=\int_0^1\Si{x}^{p\alpha}|\Ha{}{f,x}|^p\,dx.
	\end{displaymath}
	Пусть $0\le x\le1$.
	Выберем число~$\gamma$ такое, что
	\begin{displaymath}
		\max\left\{\alpha-3+\frac1p,-2-\frac1p\right\}
		<\gamma
		<\min\{0,\alpha-2\}.
	\end{displaymath}
	Применяя к внешнему интегралу неравенство Гельдера,
	учитывая, что $\gamma>-2-\frac1p$,
	получаем
	\begin{multline*}
		|\Ha{}{f,x}|^p
		\le\Cn\int_0^x(1-y)^{p\gamma}
			\left\{\int_y^1(1-z)^2
				\left|f(z)-\frac{c_1}{c_0}\right|\,dz
				\right\}^p\,dy\\
			\times\left\{
				\int_0^x(1-y)^{(-3-\gamma)\frac p{p-1}}\,dy
			\right\}^{p-1}
		\le\Cn(1-x)^{p(-2-\gamma)-1}\\
			\times\int_0^x(1-y)^{p\gamma}\left\{\int_y^1(1-z)^2
				\left|f(z)-\frac{c_1}{c_0}\right|\,dz\right\}^p\,dy.
	\end{multline*}
	Применяя теперь неравенство Гельдера к внутреннему интегралу,
	учитывая, что $\gamma>\alpha-3+\frac1p$,
	находим, что
	\begin{multline*}
		|\Ha{}{f,x}|^p
		\le\lastC(1-x)^{p(-2-\gamma)-1}\int_0^x(1-y)^{p\gamma}\\
			\times\int_y^1(1-z)^{p(\alpha-\gamma)}
				\left|f(z)-\frac{c_1}{c_0}\right|^p\,dz
				\left\{\int_y^1(1-z)^{(2-\alpha+\gamma)
					\frac p{p-1}}\,dz
				\right\}^{p-1}\,dy\\
		\le\Cn(1-x)^{p(-2-\gamma)-1}
			\int_0^x(1-y)^{p\gamma}\int_y^1(1-z)^{p(\alpha-\gamma)}
				\left|f(z)-\frac{c_1}{c_0}\right|^p\,dz\,dy.
	\end{multline*}
	Отсюда получаем, что
	\begin{multline*}
		I_3
		\le\lastC\int_0^1(1-x)^{p(\alpha-2-\gamma)-1}
		  	\int_0^x(1-y)^{p\gamma}\\
		  \times\int_y^1(1-z)^{p(\alpha-\gamma)}
			\left|f(z)-\frac{c_1}{c_0}\right|^p\,dz\,dy\,dx.
	\end{multline*}
	Поменяв пределы интегрирования,
	учитывая,
	что $\gamma<\alpha-2$ и $\gamma<0$,
	имеем
	\begin{multline*}
		I_3
		\le\lastC\int_0^1(1-z)^{p(\alpha-\gamma)}
				\left|f(z)-\frac{c_1}{c_0}\right|^p
				\int_0^z(1-y)^{p\gamma}\\
			\times\int_y^1(1-x)^{p(\alpha-2-\gamma)-1}\,dx\,dy\,dz\\
		\le\Cn\int_0^1(1-z)^{p\alpha}z
			\left|f(z)-\frac{c_1}{c_0}\right|^p\,dz
		\le\lastC\norm{f-\frac{c_1}{c_0}}.
	\end{multline*}
	Учитывая, что $f\in\Lp$ и $\alpha>-\frac1p$
	имеем
	\begin{displaymath}
		I_3<\infty.
	\end{displaymath}
	
	Обозначим
	\begin{displaymath}
		I_4=\int_{-1}^0\Si{x}^{p\alpha}|\Ha{}{f,x}|^p\,dx.
	\end{displaymath}
	Учитывая равенство~\eqref{eq:y-1-fz} имеем
	\begin{displaymath}
		\Ha{}{f,x}=\int_0^x\Si{y}^{-3}
		\int_{-1}^y\Si{z}^2\left(f(z)-\frac{c_1}{c_0}\right)\,dz\,dy.
	\end{displaymath}
	Отсюда, при $-1\le x\le0$ имеем
	\begin{displaymath}
		|\Ha{}{f,x}|^p\le\Cn\int_x^0(1+y)^{-3}
		\int_{-1}^y(1+z)^2\left|f(z)-\frac{c_1}{c_0}\right|\,dz\,dy.
	\end{displaymath}
	
	Рассуждая как и при оценке~$I_3$,
	а именно применяя дважды неравенство Гельдера,
	потом меняя пределы интегрирования,
	получим, что
	\begin{displaymath}
		I_4<\infty.
	\end{displaymath}
	
	Теперь
	\begin{displaymath}
		I=\norm{\Ha{}{f,x}}^p=I_3+I_4<\infty.
	\end{displaymath}
	
	Таким образом доказано,
	что при $1\le p<\infty$
	$\Ha{}{f,x}\in\Lp$.
	
	Пусть $p=\infty$.
	Обозначим
	\begin{displaymath}
		J=\max_{-1\le x\le1}\Si{x}^\alpha|\Ha{}{f,x}|.
	\end{displaymath}
	
	Пусть
	\begin{displaymath}
		J_1=\max_{0\le x\le1}\Si{x}^\alpha|\Ha{}{f,x}|.
	\end{displaymath}
	Тогда
	\begin{multline*}
		J_1
		\le\max_{0\le x\le1}\Si{x}^\alpha\int_0^x\Si{y}^{-3}
			\int_y^1\Si{z}^2
				\left|f(z)-\frac{c_1}{c_0}\right|\,dz\,dy\\
		\le\left\|f-\frac{c_1}{c_0}\right\|_{\infty,\alpha}
			\max_{0\le x\le1}\Si{x}^\alpha\int_0^x\Si{y}^{-3}
			\int_y^1\Si{z}^{2-\alpha}\,dz\,dy.
	\end{multline*}
	
	Учитывая, что $f\in L_{\infty,\alpha}$,
	имеем
	\begin{displaymath}
		J_1\le\Cn\max_{0\le x\le1}(1-x)^\alpha\int_0^x(1-y)^{-3}
		\int_y^1(1-z)^{2-\alpha}\,dz\,dy.
	\end{displaymath}
	Отсюда, при $0\le\alpha<3$,
	находим
	\begin{displaymath}
		J_1\le\Cn\max_{0\le x\le1}(1-x)^\alpha
		\int_0^x(1-y)^{-\alpha}\,dy<\infty.
	\end{displaymath}
	
	Пусть
	\begin{displaymath}
		J_2=\max_{-1\le x\le0}\Si{x}^\alpha|\Ha{}{f,x}|.
	\end{displaymath}
	Тогда по аналогии с оценкой для~$J_1$,
	учитывая равенство~\eqref{eq:y-1-fz},
	имеем
	\begin{displaymath}
		J_2
		\le\left\|f-\frac{c_1}{c_0}\right\|_{\infty,\alpha}
			\max_{-1\le x\le0}(1+x)^\alpha\int_x^0(1+y)^{-3}
			\int_{-1}^y(1+z)^{2-\alpha}\,dz\,dy<\infty.
	\end{displaymath}
	
	Таким образом, для $p=\infty$ вытекает,
	что
	\begin{displaymath}
		J=\max\left\{J_1,J_2\right\}<\infty,
	\end{displaymath}
	т.е. $\Ha{}{f,x}\in L_{\infty,\alpha}$.
	
	Лемма~\ref{lm:H-Lp} полностью доказана.
\end{proof}

\begin{lmm}\label{lm:dH-Lp}
	Пусть даны числа~$p$ и~$\alpha$ такие,
	что $\allp$;
	\begin{alignat*}2
		-\frac1p &<\alpha<3-\frac1p
			&\quad &\text{при $1\le p<\infty$},\\
		0        &\le\alpha<3
			&\quad &\text{при $p=\infty$}.
	\end{alignat*}
	Тогда если $f\in\Lp$,
	то $\frac d{dx}\Ha{}{f,x}\in\Lp$.
\end{lmm}

\begin{proof}
	Из определения $\Ha{}{f,x}$ имеем
	\begin{displaymath}
		\frac d{dx}\Ha{}{f,x}
		=-\Si{x}^{-3}\int_x^1\Si{z}^2
			\left(f(z)-\frac{c_1}{c_0}\right)\,dz.
	\end{displaymath}
	
	Пусть $1\le p<\infty$
	и
	\begin{multline*}
		I=\norm{\frac d{dx}\Ha{}{f,x}}^p\\
		=\int_{-1}^1\Si{x}^{\alpha-3}\left|\int_x^1\Si{z}^2
			\left(f(z)-\frac{c_1}{c_0}\right)\,dz\right|^p\,dx.
	\end{multline*}
	Сначала рассмотрим случай $p=1$.
	Рассмотрим
	\begin{multline*}
		I_1=\int_0^1\Si{x}^{\alpha-3}\left|\int_x^1\Si{z}^2
			\left(f(z)-\frac{c_1}{c_0}\right)\,dz\right|\,dx\\
		\le\Cn\int_0^1(1-x)^{\alpha-3}
			\int_x^1(1-z)^2\left|f(z)-\frac{c_1}{c_0}\right|\,dz\,dx.
	\end{multline*}
	Меняя пределы интегрирования,
	при $-1<\alpha<2$
	и $f\in L_{1,\alpha}$,
	имеем
	\begin{multline*}
		I_1
		\le\lastC\int_0^1(1-z)^2
			\left|f(z)-\frac{c_1}{c_0}\right|
			\int_0^z(1-x)^{\alpha-3}\,dx\,dz\\
		\le\Cn\int_0^1(1-z)^\alpha
			\left|f(z)-\frac{c_1}{c_0}\right|\,dz
		\le\lastC
			\left\|
				f(z)-\frac{c_1}{c_0}
			\right\|_{1,\alpha}<\infty.
	\end{multline*}
	
	Пусть
	\begin{displaymath}
		I_2
		=\int_{-1}^0\Si{x}^{\alpha-3}\left|\int_x^1\Si{z}^2
			\left(f(z)-\frac{c_1}{c_0}\right)\,dz\right|\,dx.
	\end{displaymath}
	Аналогично как и при оценке~$I_1$,
	учитывая равенство~\eqref{eq:y-1-fz},
	получим
	\begin{displaymath}
		I_2<\infty.
	\end{displaymath}
	
	Из того, что $I_1<\infty$ и $I_2<\infty$,
	следует, что
	\begin{displaymath}
		I=I_1+I_2<\infty,
	\end{displaymath}
	т.е. $\frac d{dx}\Ha{}{f,x}\in L_{1,\alpha}$.
	
	Пусть $1<p<\infty$.
	Рассмотрим
	\begin{multline*}
		I_3=\int_0^1\Si{x}^{p(\alpha-3)}\left|\int_x^1\Si{z}^2
			\left(f(z)-\frac{c_1}{c_0}\right)\,dz\right|^p\,dx\\
		\le\Cn\int_0^1(1-x)^{p(\alpha-3)}\left\{\int_x^1(1-z)^2
			\left|f(z)-\frac{c_1}{c_0}\right|\,dz\right\}^p\,dx.
	\end{multline*}
	Пусть $\alpha<\gamma<3-\frac1p$.
	Применяя неравенство Гельдера,
	потом меняя пределы интегрирования,
	получаем, что
	\begin{multline*}
		I_3\le\lastC\int_0^1(1-x)^{p(\alpha-3)}
			\int_x^1(1-z)^{p\gamma}
				  \left|f(z)-\frac{c_1}{c_0}\right|^p\,dz\\
			\times\left\{
				\int_x^1(1-z)^{(2-\gamma)\frac p{p-1}}\,dz
			\right\}^{p-1}\,dx\\
		=\Cn\int_0^1(1-x)^{p(\alpha-\gamma)-1}
			\int_x^1(1-z)^{p\gamma}
				\left|f(z)-\frac{c_1}{c_0}\right|^p\,dz\,dx\\
		=\lastC\int_0^1(1-z)^{p\gamma}
			\left|f(z)-\frac{c_1}{c_0}\right|^p
				\int_0^z(1-x)^{p(\alpha-\gamma)-1}\,dx\,dz\\
		=\Cn\int_0^1(1-z)^{p\alpha}
			\left|f(z)-\frac{c_1}{c_0}\right|^p\,dz
		\le\lastC\norm{f-\frac{c_1}{c_0}}.
	\end{multline*}
	Отсюда, учитывая,
	что $f\in\Lp$ и $\alpha>-\frac1p$,
	имеем
	\begin{displaymath}
		I_3<\infty.
	\end{displaymath}
	
	Рассмотрим
	\begin{displaymath}
		I_4
		=\int_{-1}^0\Si{x}^{p(\alpha-3)}\left|\int_x^1\Si{z}^2
			\left(f(z)-\frac{c_1}{c_0}\right)\,dz\right|^p\,dx.
	\end{displaymath}
	
	Аналогично как и при оценке~$I_3$,
	учитывая равенство~\eqref{eq:y-1-fz},
	получим
	\begin{displaymath}
		I_4<\infty.
	\end{displaymath}
	
	Теперь
	\begin{displaymath}
		I=I_3+I_4<\infty.
	\end{displaymath}
	
	Таким образом доказано,
	что при $1\le p<\infty$
	$\frac d{dx}\Ha{}{f,x}\in\Lp$.
	
	Пусть теперь $p=\infty$.
	Рассмотрим
	\begin{multline*}
		J=\max_{-1\le x\le1}\Si{x}^\alpha
		  \left|\frac d{dx}\Ha{}{f,x}\right|\\
		=\max_{-1\le x\le1}\Si{x}^{\alpha-3}
		  \left|
			\int_x^1\Si{z}^2
			  \left(f(z)-\frac{c_1}{c_0}\right)\,dz
			\right|.
	\end{multline*}
	
	При $\alpha<3$ имеем
	\begin{multline*}
		J_1=\max_{0\le x\le1}\Si{x}^{\alpha-3}
		  \left|
			\int_x^1\Si{z}^2
				  \left(f(z)-\frac{c_1}{c_0}\right)\,dz
		  \right|\\
		\le\left\|f-\frac{c_1}{c_0}\right\|_{\infty,\alpha}
			\max_{0\le x\le1}(1-x)^{\alpha-3}
			\int_x^1(1-z)^{2-\alpha}\,dz
		=\Cn\left\|f-\frac{c_1}{c_0}\right\|_{\infty,\alpha}.
	\end{multline*}
	Отсюда, при $f\in L_{\infty,\alpha}$ и $\alpha\ge0$
	имеем
	\begin{displaymath}
		J_1<\infty.
	\end{displaymath}
	
	Аналогично, с использованием равенства~\eqref{eq:y-1-fz},
	при $f\in L_{\infty,\alpha}$ и $0\le\alpha<3$
	находим,
	что
	\begin{displaymath}
		J_2=\max_{-1\le x\le0}\Si{x}^{\alpha-3}
			\left|
			  \int_x^1\Si{z}^2
				\left(f(z)-\frac{c_1}{c_0}\right)\,dz
			\right|
		<\infty.
	\end{displaymath}
	
	Таким образом
	\begin{displaymath}
		J=\max\{J_1,J_2\}<\infty,
	\end{displaymath}
	значит $\frac d{dx}\Ha{}{f,x}\in L_{\infty,\alpha}$.
	
	Лемма~\ref{lm:dH-Lp} доказана.
\end{proof}

\begin{lmm}\label{lm:DxH}
	Пусть $f\in\Lmu$.
	Справедливые следующие равенства
	\begin{displaymath}
		\Dx^l\Ha r{f,x}=\Ha{r-l}{f,x}-\frac{c_{r-l+1}}{c_0}
		\quad (l=1,\dots,r-1)
	\end{displaymath}
	и
	\begin{equation}\label{eq:DxHa}
		\Dx^r\Ha r{f,x}=f(x)-\frac{c_1}{c_0},
	\end{equation}
	где $c_{r-l+1}=\int_{-1}^1\Si{z}^2\Ha{r-l}{f,z}\,dz$.
\end{lmm}

\begin{proof}
	Докажем сначала равенство~\eqref{eq:DxHa}.
	Для $r=1$ имеем
	\begin{displaymath}
		\Dx\Ha{}{f,x}=f(x)-\frac{c_1}{c_0}.
	\end{displaymath}
	
	Теперь, учитывая,
	что по утверждении леммы~\ref{lm:H-Lp}
	следует $\Ha r{f,x}\in\Lmu$,
	равенство~\eqref{eq:DxHa}
	доказывается по индукции.
	
	Из доказанного равенства~\eqref{eq:DxHa}
	следует,
	что для $l=1,\dots,r-1$
	имеем
	\begin{displaymath}
		\Dx^l\Ha r{f,x}=\Dx^l\Ha l{\Ha{r-l}{f,x},x}
		=\Ha{r-l}{f,x}-\frac{c_{r-l+1}}{c_0}.
	\end{displaymath}
	
	Лемма~\ref{lm:DxH} доказана.
\end{proof}

\begin{lmm}\label{lm:H-AD}
	Пусть даны числа~$p$, $\alpha$ и~$r$
	такие,
	что $\allp$,
	$r\in\numN$;
	\begin{alignat*}2
		-\frac1p &<\alpha<3-\frac1p &\quad &\text{при $1\le p<\infty$},\\
		0        &\le\alpha<3       &\quad &\text{при $p=\infty$}.
	\end{alignat*}
	Тогда если $f\in\Lp$,
	то $\Ha r{f,x}\in\AD$.
\end{lmm}

\begin{proof}
	По лемме~\ref{lm:H-Lp} имеем,
	что $\Ha r{f,x}\in\Lp$.
	Из условий леммы следует,
	что $f\in\Lmu$ и $\Ha r{f,x}\in\Lmu$.
	Значит, постоянные~$c_r$ $(r=1,2,\dotsc)$
	в определении $\Ha r{f,x}$
	определены.
	
	Рассмотрим сначала случай $r=1$.
	По определению оператора $\Ha{}{f,x}$ ясно,
	что $\frac d{dx}\Ha{}{f,x}$
	--- абсолютно непрерывная  функция
	на каждом отрезке $[a,b]\subset(-1,1)$.
	Далее,
	из леммы~\ref{lm:DxH} вытекает,
	что
	\begin{displaymath}
		\Dx\Ha{}{f,x}=f(x)-\frac{c_1}{c_0},
	\end{displaymath}
	и следовательно $\Dx\Ha{}{f,x}\in\Lp$.
	Из лемм~\ref{lm:H-Lp} и~\ref{lm:dH-Lp}
	следует,
	что $\Ha{}{f,x}\in\Lp$.
	Таким образом, $\Ha{}{f,x}\in AD^1(p,\alpha)$.
	
	Теперь,
	применяя формулу Лейбница и Лемму~\ref{lm:DxH},
	утверждение леммы доказывается на основании индукции.
\end{proof}

\begin{lmm}\label{lm:Hd}
	Пусть $f\in\Lmu$.
	Справедливы равенства
	\begin{equation}\label{eq:Hd}
		\Hd r{f,x}
		=\frac1{\kap^r}\Dl\delta r{\Ha r{f,x},x}+\frac{c_1}{c_0}
		\quad (r=1,2,\dotsc),
	\end{equation}
	где
	\begin{displaymath}
		\kap=\int_0^\delta\sincosv\int_0^v\sincosu\,du\,dv.
	\end{displaymath}
\end{lmm}

\begin{proof}
	Докажем сначала равенство~\eqref{eq:Hd}
	для $r=1$.
	По лемме~\ref{lm:DxH} имеем
	\begin{displaymath}
		f(x)=\Dx\Ha{}{f,x}+\frac{c_1}{c_0}.
	\end{displaymath}
	Поэтому
	\begin{multline*}
		\Hd{}{f,x}=\frac1\kap\int_0^\delta\sincosv\\
		\times\int_0^v\hatT u{}{\Dx\Ha{}{f,x}
			+\frac{c_1}{c_0},x}\sincosu\,du\,dv.
	\end{multline*}
	Поскольку из свойств оператора $\hatT u{}{f,x}$,
	отмеченных в лемме~\ref{lm:properties-tau},
	следует, что
	\begin{displaymath}
		\hatT u{}{\Dx\Ha{}{f,x}+\frac{c_1}{c_0},x}
		=\hatT u{}{\Dx\Ha{}{f,x},x}+\frac{c_1}{c_0},
	\end{displaymath}
	то
	\begin{multline*}
		\Hd{}{f,x}=\frac1\kap\int_0^\delta\sincosv\\
		\times\int_0^v\hatT u{}{\Dx\Ha{}{f,x},x}\sincosu\,du\,dv
			+\frac{c_1}{c_0}.
	\end{multline*}
	Применяя лемму~\ref{lm:tauuDx},
	учитывая, что по лемме~\ref{lm:H-AD}
	функция $\Ha{}{f,x}$ имеет абсолютно непрерывную
	на каждом отрезке $[a,b]\subset(-1,1)$
	производную $\frac d{dx}\Ha{}{f,x}$,
	получаем
	\begin{displaymath}
		\Hd{}{f,x}
		=\frac1\kap\Dl\delta{}{\Ha{}{f,x},x}+\frac{c_1}{c_0}.
	\end{displaymath}
	
	Теперь для любого натурального~$r$
	справедливость равенства~\eqref{eq:Hd}
	доказывается индукции,
	применяя леммы~\ref{lm:DxH}, \ref{lm:cr:Dtau}
	и~\ref{lm:tauuDx}.
\end{proof}

\begin{cor}
	Пусть $f\in\Lmu$,
	тогда справедливы равенства
	\begin{displaymath}
		\Dx^r\Hd r{f,x}=\frac1{\kap^r}\Dl\delta r{f,x}
		\quad (r=1,2,\dotsc).
	\end{displaymath}
\end{cor}

\begin{proof}
	По лемме~\ref{lm:Hd} имеем
	\begin{displaymath}
		\Hd r{f,x}
		=\frac1{\kap^r}\Dl\delta r{\Ha r{f,x},x}
			+\frac{c_1}{c_0}.
	\end{displaymath}
	Так как, из леммы~\ref{lm:H-AD} вытекает $\Ha r{f,x}\in\AD$,
	то по лемме~\ref{lm:cr:Dtau}
	получаем, что
	\begin{multline*}
		\Dx^r\Hd r{f,x}
		  =\frac1{\kap^r}\Dx^r\Dl\delta r{\Ha r{f,x},x}\\
		=\frac1{\kap^r}\Dl\delta r{\Dx^r\Ha r{f,x},x}.
	\end{multline*}
	Применяя лемму~\ref{lm:DxH},
	находим
	\begin{displaymath}
		\Dx^r\Hd r{f,x}
		=\frac1{\kap^r}\Delta_\delta^r
		  \Bigl(f-\frac{c_1}{c_0},x\Bigr)
		=\frac1{\kap^r}\Dl\delta r{f,x}.
	\end{displaymath}
	
	Следствие доказано.
\end{proof}

\begin{lmm}\label{lm:Hd-AD}
	Пусть даны числа~$p$, $\alpha$, $r$ и~$\delta$
	такие,
	что $\allp$, $r\in\numN$, $0\le\delta<\pi$;
	\begin{alignat*}2
		\frac12      &<\alpha\le1
		  &\quad &\text{при $p=1$},\\
		1-\frac1{2p} &<\alpha<\frac32-\frac1{2p}
		  &\quad &\text{при $1<p<\infty$},\\
		1            &\le\alpha<\frac32
		  &\quad &\text{при $p=\infty$}.
	\end{alignat*}
	Если $f\in\Lp$,
	то $\Hd r{f,x}\in\AD$.
\end{lmm}

\begin{proof}
	Так как, в условиях леммы имеем $f\in\Lmu$,
	то по лемме~\ref{lm:Hd}
	\begin{displaymath}
		\Hd r{f,x}
		=\frac1{\kap^r}\Dl\delta r{\Ha r{f,x},x}
		  +\frac{c_1}{c_0}.
	\end{displaymath}
	По лемме~\ref{lm:H-AD}
	$\Ha r{f,x}\in\AD$.
	Из леммы~\ref{lm:Dtau} следует,
	что $\Hd r{f,x}$ имеет абсолютно непрерывную
	на каждом отрезке $[a,b]\subset(-1,1)$
	производную
	$\frac{d^{2r-1}}{dx^{2r-1}}\Hd r{f,x}$.
	Применяя теорему Лебега о предельном переходе под знаком интеграла,
	имеем, что для $l=1,\dots,r$
	\begin{displaymath}
		\Dx^l\Hd r{f,x}
		=\frac1{\kap^r}\Dl\delta r{\Dx^l\Ha r{f,x},x},
	\end{displaymath}
	т.е., опять по леммам~\ref{lm:Hd} и~\ref{lm:Dtau},
	$\Dx^l\Hd r{f,x}$ --- абсолютно непрерывная функция
	на каждом отрезке $[a,b]\subset(-1,1)$.
	
	Из леммы~\ref{lm:DxH}
	для $l=1,\dots,r$
	имеем
	\begin{displaymath}
		\Dx^l\Hd r{f,x}
		=\frac1{\kap^r}\Dl\delta r{\Ha{r-l}{f,x},x}.
	\end{displaymath}
	Теперь,
	применяя лемму~\ref{lm:bound-tau}
	при фиксированном~$\delta$,
	учитывая, что по лемме~\ref{lm:H-Lp}
	$\Ha{r-l}{f,x}\in\Lp$,
	имеем что
	$\Dx^l\Hd r{f,x}\in\Lp$.
	
	Следовательно, $\Hd r{f,x}\in\AD$.
	Тем самым лемма~\ref{lm:Hd-AD} доказана.
\end{proof}

\begin{lmm}\label{lm:E-D}
	Пусть даны числа~$p$, $\alpha$ и~$r$ такие,
	что $\allp$,
	$r\in\numN$;
	\begin{alignat*}2
		-\frac12    &<\alpha\le2
		  &\quad &\text{при $p=1$},\\
		-\frac1{2p} &<\alpha<\frac52-\frac1{2p}
		  &\quad &\text{при $1<p<\infty$},\\
		0           &\le\alpha<\frac52
		  &\quad &\text{при $p=\infty$}.
	\end{alignat*}
	Пусть $f\in\AD$.
	Тогда справедливо неравенство
	\begin{displaymath}
		\E\le C\frac1{n^{2r}}\norm{\Dx^r f(x)},
	\end{displaymath}
	где постоянная~$C$ не зависит от~$f$ и~$n$.
\end{lmm}

\begin{proof}
	Для $r=1$ лемма доказана
	в работе~\cite{p-berisha:anal-99}.
	
	Пусть~$P_n(x)$ --- алгебраический многочлен
	наилучшего приближения функции $\Dx f(x)$,
	степени не выше, чем $n-1$.
	Ясно, что многочлен~$P_n(x)$ можно представить в виде
	\begin{displaymath}
		P_n(x)=\sum_{k=0}^{n-1}\lambda_k\Px k(x).
	\end{displaymath}
	Пусть
	\begin{displaymath}
		g(x)=f(x)
		  +\sum_{k=0}^{n-1}\frac{\lambda_k}{k(k+5)}\Px k(x).
	\end{displaymath}
	Тогда по уже доказанному для $r=1$ случаю леммы
	имеем~\cite[с.171]{erdelyi-m-o-t:transcendental}
	\begin{multline*}
		\Epar n g\le\Cn\frac1{n^2}\norm{\Dx g(x)}\\
		=\lastC\frac1{n^2}
			\biggl\|
			  \Dx f(x)-\sum_{k=0}^{n-1}\lambda_k\Px k(x)
			\biggr\|\\
			=\lastC\frac1{n^2}\Epar n{\Dx f}.
	\end{multline*}
	Отсюда,
	учитывая, что $f(x)-g(x)$ --- алгебраический многочлен
	степени не выше, чем $n-1$,
	получаем
	\begin{displaymath}
		\E\le\Epar n{f-g}+\Epar n g
		\le\lastC\frac1{n^2}\Epar n{\Dx f}.
	\end{displaymath}
	
	Теперь, применяя это неравенство $r$~раз,
	получим, что
	\begin{displaymath}
		\E\le\Cn\frac1{n^{2r}}\Epar n{\Dx^r f}
		\le\lastC\frac1{n^{2r}}\norm{\Dx^r f(x)}.
	\end{displaymath}
	
	Лемма~\ref{lm:E-D} доказана.
\end{proof}

\section{Основные утверждения}

\begin{thm}\label{th:w-K}
	Пусть даны числа~$p$, $\alpha$ и~$r$ такие,
	что $\allp$,
	$r\in\numN$;
	\begin{alignat*}2
		\frac12      &<\alpha\le1
		  &\quad &\text{при $p=1$},\\
		1-\frac1{2p} &<\alpha<\frac32-\frac1{2p}
		  &\quad &\text{при $1<p<\infty$},\\
		1            &\le\alpha<\frac32
		  &\quad &\text{при $p=\infty$}.
	\end{alignat*}
	Пусть $f\in\Lp$.
	Тогда при всех $\delta\in[0,\pi)$
	имеют место неравенства
	\begin{displaymath}
		\Cn\left(\Co\delta\right)^{r(r-1)}\K\le\w
		\le\Cn\frac1{\left(\Co\delta\right)^r}\K,
	\end{displaymath}
	где положительные постоянные~$\prevC$ и~$\lastC$
	не зависят от~$f$ и~$\delta$.
\end{thm}

\begin{proof}
	Для любой функции $g(x)\in\AD$ имеем
	\begin{displaymath}
		\w\le\wpar{f-g,\delta}+\wpar{g,\delta}.
	\end{displaymath}
	Применяя лемму~\ref{lm:bound-tau},
	находим, что
	\begin{displaymath}
		\wpar{f-g,\delta}
		\le\Cn\frac1{\left(\Co\delta\right)^r}\norm{f-g}.
	\end{displaymath}
	Далее, в силу леммы~\ref{lm:w-Dx}
	\begin{displaymath}
		\wpar{g,\delta}\le\Cn\frac1{\left(\Co\delta\right)^r}
		\delta^{2r}\norm{\Dx^r g(x)}.
	\end{displaymath}
	
	Поэтому
	\begin{displaymath}
		\w\le\Cn\frac1{\left(\Co\delta\right)^r}
		\left(\norm{f-g}+\delta^{2r}\norm{\Dx^r g(x)}\right).
	\end{displaymath}
	Переходя в этом неравенстве к точной нижней грани
	по $g(x)\in\AD$,
	получаем правое неравенство теоремы.
	
	Для доказательства левого неравенства
	для данной функции $f\in\Lp$
	рассмотрим функцию
	\begin{displaymath}
		\gd=\left(1-(1-H_\delta^r)^r\right)(f,x),
	\end{displaymath}
	где $1(f,x)=f(x)$.
	
	Из леммы~\ref{lm:Hd-AD} следует, что
	$\Hd l{f,x}\in AD^l(p,\alpha)$ $(l\in\numN)$.
	Поскольку
	\begin{displaymath}
		1-(1-H_\delta^r)^r
		=\sum_{k=1}^r\binom rk(-1)^k H_\delta^{kr},
	\end{displaymath}
	то, учитывая,
	что $AD^{kr}(p,\alpha)\subseteq\AD$ $(k=1,\dots,r)$,
	получаем, что
	\begin{displaymath}
		\gd\in\AD.
	\end{displaymath}
	
	Оценим выражение
	\begin{displaymath}
		\norm{\Dx^r\gd}.
	\end{displaymath}
	Для этого, замечаем, что,
	поскольку
	$\Hd{kr-l}{f,x}$
	$(k=2,\dots,r;\;\allowbreak l=0,1,\allowbreak\dots,r-1)$
	имеет на каждом отрезке $[a,b]\subset(-1,1)$
	абсолютно непрерывную $2r-1$ производную,
	то применяя сначала теорему Лебега
	о предельном переходе под знаком интеграла,
	потом лемму~\ref{lm:cr:Dtau},
	обобщенное неравенство Минковского
	и, наконец, лемму~\ref{lm:bound-tau},
	получаем, что
	\begin{multline*}
		\norm{\Dx^r\Hd{kr}{f,x}}
		\le\frac1{\kap}\int_0^\delta\sincosv\\
			\times\int_0^v
				\norm{\hatT u{}{\Dx^r\Hd{kr-1}{f,x},x}}
				\sincosu\,du\,dv\\
		\le\Cn\frac1{\Co\delta}\norm{\Dx^r\Hd{kr-1}{f,x}}.
	\end{multline*}
	
	Применяя это неравенство $k-1$ раз, получим, что
	\begin{displaymath}
		\norm{\Dx^r\Hd{kr}{f,x}}
		\le\Cn\frac1{\left(\Co\delta\right)^{r(k-1)}}
		  \norm{\Dx^r\Hd r{f,x}}.
	\end{displaymath}
	
	Так как, $\gd$ представляет собой сумму членов
	содержащих $\Hd{kr}{f,x}$ $(k=1,\dots,r)$,
	то по последнему неравенству находим
	\begin{displaymath}
		\norm{\Dx^r\gd}
		\le\Cn\frac1{\left(\Co\delta\right)^{r(r-1)}}
		  \norm{\Dx^r\Hd r{f,x}}.
	\end{displaymath}
	Применяя следствие из леммы~\ref{lm:Hd},
	получаем
	\begin{displaymath}
		\norm{\Dx^r\gd}
		\le\lastC\frac1{\kap^r\left(\Co\delta\right)^{r(r-1)}}
		\norm{\Dl\delta r{f,x}}.
	\end{displaymath}
	
	Легко оценить, что при $0<\delta\le\frac\pi2$
	\begin{displaymath}
		\kap\ge\Cn\delta^2.
	\end{displaymath}
	Отсюда следует, что при $0<\delta\le\frac\pi2$
	\begin{equation}\label{eq:Dx-w}
		\delta^{2r}\norm{\Dx^r \gd}
		\le\Cn\frac1{\left(\Co\delta\right)^{r(r-1)}}\w.
	\end{equation}
	
	С другой стороны
	\begin{multline}\label{eq:f-gd-Hd}
		\norm{f(x)-\gd}
		=\norm{f(x)-\left(1-(1-H_\delta^r)^r\right)(f,x)}\\
		=\norm{(1-H_\delta^r)^r(f,x)}.
	\end{multline}
	Заметим, что
	\begin{equation}\label{eq:1-Hd}
		1-H_\delta^r
		=(1-H_\delta)(1+H_\delta+H_\delta^2+\dots+H_\delta^{r-1}).
	\end{equation}
	Теперь, применяя обобщенное неравенство Минковского
	и лемму~\ref{lm:bound-tau},
	для $l=1,\dots,r-1$, имеем
	\begin{multline*}
		\norm{\Hd l{f,x}}\le\frac1{\kap}\int_0^\delta\sincosv\\
		\times\int_0^v\norm{\hatT u{}{\Hd{l-1}{f,x},x}}
			\sincosu\,du\,dv\\
		\le\Cn\frac1{\Co\delta}\norm{\Hd{l-1}{f,x}}.
	\end{multline*}
	Применяя это неравенство $l$~раз, получаем, что
	\begin{displaymath}
		\norm{\Hd l{f,x}}
		\le\Cn\frac1{\left(\Co\delta\right)^l}\norm{f(x)}.
	\end{displaymath}
	Поэтому, из равенства~\eqref{eq:1-Hd},
	применяя обобщенное неравенство Минковского,
	имеем, что
	\begin{multline}\label{eq:1-Hd-Dl}
		\norm{(1-H_\delta^r)(f,x)}
		\le\Cn\frac1{\left(\Co\delta\right)^{r-1}}
			\norm{(1-H_\delta)(f,x)}\\
		\le\lastC\frac1{\kap\left(\Co\delta\right)^{r-1}}
			\int_0^\delta\sincosv\\
			\times\int_0^v\norm{\hatT u{}{f,x}-f(x)}\sincosu\,du\,dv\\
		\le\Cn\frac1{\left(\Co\delta\right)^{r-1}}
			\sup_{0\le u\le\delta}\norm{\Dl u{}{f,x}}.
	\end{multline}
	
	Применяя неравенство~\eqref{eq:1-Hd-Dl},
	из равенства~\eqref{eq:f-gd-Hd} получим,
	что
	\begin{multline}\label{eq:f-gd-Dl}
		\norm{f(x)-\gd}\\
		\le\lastC\frac1{\left(\Co\delta\right)^{r-1}}
			\sup_{0\le t_1\le\delta}
				\norm{\Dl{t_1}{}{(1-H_\delta^r)^{r-1}(f,x),x}}.
	\end{multline}
	
	Замечаем, что меняя пределы интегрирования получаем
	\begin{displaymath}
		\hatT t{}{\Hd{}{f,x},x}=\Hd{}{\hatT t{}{f,x},x}.
	\end{displaymath}
	Применяя это равенство $r$~раз получим
	\begin{displaymath}
		\hatT t{}{\Hd r{f,x},x}=\Hd r{\hatT t{}{f,x},x}.
	\end{displaymath}
	Отсюда очевидно, что
	\begin{displaymath}
		\Dl t{}{(1-H_\delta^r)(f,x),x}=(1-H_\delta^r)(\Dl t{}{f,x},x).
	\end{displaymath}
	
	Применяя сначала это равенство,
	затем неравенство~\eqref{eq:f-gd-Dl},
	потом неравенство~\eqref{eq:1-Hd-Dl},
	получим, что
	\begin{multline*}
		\norm{f(x)-\gd}\\
		\le\Cn\frac1{\left(\Co\delta\right)^{2(r-1)}}
			\sup_{0\le t_1\le\delta}\sup_{0\le t_2\le\delta}
				\norm{\Dl{t_1,t_2}2{(1-H_\delta^r)^{r-2}(f,x),x}}.
	\end{multline*}
	
	Теперь, применяя $r-1$ раз эту процедуру, получим, что
	\begin{multline*}
		\norm{f(x)-\gd}
		\le\Cn\frac1{\left(\Co\delta\right)^{r(r-1)}}
			\sup_{\substack{0\le t_i\le\delta\\ i=1,\dots,r}}
				\norm{\Dl{\arr t r}r{f,x}}\\
		\le\lastC\frac1{\left(\Co\delta\right)^{r(r-1)}}\w.
	\end{multline*}
	
	Таким образом, для $0<\delta\le\frac\pi2$,
	из этого неравенства и неравенства~\eqref{eq:Dx-w}
	следует, что
	\begin{multline*}
		I_\delta=\norm{f(x)-\gd}+\delta^{2r}\norm{\Dx^r\gd}\\
		\le\Cn\frac1{\left(\Co\delta\right)^{r(r-1)}}\w.
	\end{multline*}
	Тем самим доказано левое неравенство теоремы для
	$0<\delta\le\frac\pi2$.
	
	Поскольку для $\frac\pi2\le\delta<\pi$
	имеем
	$\delta^2<\pi^2\cdot1$ и $1<\frac\pi2$,
	то
	\begin{multline*}
		I_\delta\le\pi^{2r}
			\Bigl(
				\norm{f(x)-g_{1,r}(x)}
				+1\cdot\norm{\Dx^r g_{1,r}(x)}
			\Bigr)\\
		\le\Cn\wpar{f,1}
		\le\lastC\frac1{\left(\Co\delta\right)^{r(r-1)}}\w.
	\end{multline*}
	
	Для $\delta=0$ левое неравенство теоремы тривиально.
	
	Итак для любого $0\le\delta<\pi$
	доказано левое неравенство теоремы.
	
	Теорема~\ref{th:w-K} полностью доказана.
\end{proof}

\begin{thm}\label{th:jackson}
	Пусть даны числа~$p$, $\alpha$ и~$r$ такие,
	что $\allp$,
	$r\in\numN$;
	\begin{alignat*}2
		\frac12      &<\alpha\le1
		  &\quad &\text{при $p=1$},\\
		1-\frac1{2p} &<\alpha<\frac32-\frac1{2p}
		  &\quad &\text{при $1<p<\infty$},\\
		1            &\le\alpha<\frac32
		  &\quad &\text{при $p=\infty$}.
	\end{alignat*}
	Пусть $f\in\Lp$.
	Тогда для любого натурального~$n$
	справедливы неравенства
	\begin{displaymath}
		\Cn\E\le\wpar{f,\frac1n}
		\le\Cn\frac1{n^{2r}}\sum_{\nu=1}^n\nu^{2r-1}\Epar\nu f,
	\end{displaymath}
	где положительные постоянные~$\prevC$ и~$\lastC$
	не зависят от~$f$ и~$n$.
\end{thm}

\begin{proof}
	Для любой функции $g(x)\in\AD$,
	применяя лемму~\ref{lm:E-D},
	имеем
	\begin{displaymath}
		\E\le\Epar n{f-g}+\Epar n g
		\le\norm{f-g}+\Cn\frac1{n^{2r}}\norm{\Dx^r g(x)},
	\end{displaymath}
	где постоянная~$\lastC$ не зависит от~$g$ и~$n$.
	Отсюда, переходя к точной нижней граны
	по всем $g(x)\in\AD$,
	получим
	\begin{displaymath}
		\E\le\Cn\Kpar{f,\frac1n}.
	\end{displaymath}
	Применяя теорему~\ref{th:w-K},
	получаем, что
	\begin{displaymath}
		\E\le\Cn\left(\cos^4\frac1{2n}\right)^{-r(r-1)}
			\wpar{f,\frac1n}
		\le\Cn\wpar{f,\frac1n}.
	\end{displaymath}
	
	Левое неравенство теоремы доказано.
	
	Докажем правое неравенство теоремы.
	Пусть~$P_n(x)$ алгебраический многочлен наилучшего приближения для~$f$,
	степени не выше, чем $n-1$.
	Пусть~$k$ выбрано так, что
	\begin{displaymath}
		2^k\le n<2^{k+1}.
	\end{displaymath}
	Из теоремы~\ref{th:w-K},
	учитывая, что $P_{2^k}(x)\in\AD$,
	следует, что
	\begin{multline}\label{eq:w-DxP}
		\wpar{f,\frac1n}
		\le\Cn\left(\cos\frac1{2n}\right)^{-4r}
			\Kpar{f,\frac1n}\\
		\le\Cn
			\left(
				\Epar{2^k}f+\frac1{n^{2r}}\norm{\Dx^r P_{2^k}(x)}
			\right).
	\end{multline}
	Так как,
	\begin{displaymath}
		\Dx^r P_{2^k}(x)
		=\sum_{\nu=0}^{k-1}\Dx^r(P_{2^{\nu+1}}(x)-P_{2^\nu}(x)),
	\end{displaymath}
	то применяя $r$~раз
	следствие из леммы~\ref{lm:bernshtein-markov},
	получаем
	\begin{multline*}
		\norm{\Dx^r P_{2^k}(x)}
		\le\Cn\sum_{\nu=0}^{k-1}2^{2(\nu+1)r}
			\norm{P_{2^{\nu+1}}-P_{2^\nu}}\\
		\le2\lastC\sum_{\nu=0}^{k-1}2^{2(\nu+1)r}\Epar{2^\nu}f.
	\end{multline*}
	Поэтому, учитывая неравенство~\eqref{eq:w-DxP},
	\begin{displaymath}
		\wpar{f,\frac1n}
		\le\Cn\frac1{n^{2r}}
			\sum_{\nu=0}^k 2^{2(\nu+1)r}\Epar{2^\nu}f.
	\end{displaymath}
	
	Теперь, замечая, что для $\nu=1,\dots,k$
	\begin{displaymath}
		\sum_{j=2^{\nu-1}}^{2^\nu-1}j^{2r-1}\Epar jf
		\ge2^{2(\nu+1)r-4r}\Epar{2^\nu}f,
	\end{displaymath}
	находим
	\begin{multline*}
		\wpar{f,\frac1n}
		\le\Cn\frac1{n^{2r}}
			\biggl(
			  2^{2r}\Epar1f
			  +\sum_{\nu=1}^k
				\sum_{j=2^{\nu-1}}^{2^\nu-1}j^{2r-1}\Epar jf
			\biggr)\\
		\le\Cn\frac1{n^{2r}}\sum_{\nu=1}^n\nu^{2r-1}\Epar\nu f.
	\end{multline*}
	
	Теорема~\ref{th:jackson} доказана.
\end{proof}

\end{document}